\documentclass[12pt, reqno]{amsart}
\usepackage{graphicx}
\usepackage{hyperref}
\usepackage[caption = false]{subfig}
\usepackage{amsthm}
\usepackage{amssymb}
\usepackage{mathrsfs}
\usepackage{mathtools}
\usepackage{enumerate}
\usepackage{amssymb}
\usepackage{wrapfig}
\usepackage{float}
\allowdisplaybreaks

\usepackage[twoside,paperwidth=190mm, paperheight=297mm, top=35mm, bottom=20mm, left=20mm, right=20mm, marginparsep=3mm, marginparwidth=40mm]{geometry}

\numberwithin{equation}{section}

\theoremstyle{plain}

\newtheorem{theorem}{Theorem}[section]
\newtheorem{corollary}[theorem]{Corollary}
\newtheorem{example}[theorem]{Example}
\newtheorem{lemma}{Lemma}[section]

\theoremstyle{definition}
\newtheorem{definition}[theorem]{Definition}
\theoremstyle{remark}
\newtheorem{remark}{Remark}[section]

\makeatother

\begin{document}

\title{On Geometrical Properties of Certain Analytic functions}
	\thanks{The work of the second author is supported by University Grant Commission, New-Delhi, India  under UGC-Ref. No.:1051/(CSIR-UGC NET JUNE 2017).}	
	
	\author[S. Sivaprasad Kumar]{S. Sivaprasad Kumar}
	\address{Department of Applied Mathematics, Delhi Technological University,
		Delhi--110042, India}
	\email{spkumar@dce.ac.in}

	\author[Kamaljeet]{Kamaljeet Gangania}
	\address{Department of Applied Mathematics, Delhi Technological University,
		Delhi--110042, India}
	\email{gangania.m1991@gmail.com}

\maketitle	
	
\begin{abstract} 
	We introduce the class of analytic functions
$$\mathcal{F}(\psi):= \left\{f\in \mathcal{A}: \left(\frac{zf'(z)}{f(z)}-1\right) \prec \psi(z),\; \psi(0)=0 \right\},$$
where $\psi$ is univalent and establish the growth theorem with some geometric conditions on $\psi$ and obtain the Koebe domain with some related sharp inequalities. Note that functions in this class may not be univalent. As an application, we obtain the growth theorem for the complete range of $\alpha$ and $\beta$ for the functions in the classes 
$\mathcal{BS}(\alpha):= \{f\in \mathcal{A} : ({zf'(z)}/{f(z)})-1 \prec {z}/{(1-\alpha z^2)},\; \alpha\in [0,1) \}$ and 
$\mathcal{S}_{cs}(\beta):= \{f\in \mathcal{A} : ({zf'(z)}/{f(z)})-1 \prec {z}/({(1-z)(1+\beta z)}),\; \beta\in [0,1) \}$, respectively which improves the earlier known bounds. The sharp Bohr-radii for the classes $S(\mathcal{BS}(\alpha))$ and $\mathcal{BS}(\alpha)$ are also obtained. A few examples as well as certain newly defined classes on the basis  of geometry are also discussed.
\end{abstract}
\vspace{0.5cm}

\maketitle
	
	\section{Introduction}

Let $\mathcal{A}$ denotes the class of analytic functions of the form $f(z)=z+\sum_{k=2}^{\infty}a_kz^k$ in the open unit disk ${\Delta}:=\{z: |z|<1\}$.
Let $f(z)=w$ and $\Gamma_w$ be the image of $\Gamma_z$ (the circle $C_r: z=re^{i\theta}$) under the function $f$ in $\mathcal{A}$. The curve $\Gamma_w$ is said to be starlike with respect to $w_0=0$ if $\arg(w-w_0)$ is a non-decreasing function of $\theta$, that is,
\begin{equation*}\label{arg-def}
\frac{d}{d\theta} \arg(w-w_0)\geq0, \quad \theta\in [0,2\pi],
\end{equation*}  
which is equivalent to
\begin{equation}\label{charcter}
\frac{d}{d\theta} \arg(w-w_0)=\Re\left(\frac{zf'(z)}{f(z)}\right)\geq0.
\end{equation}
If the inequality \eqref{charcter} holds for each circle $|z|=r<1$, then it characterizes a special class $\mathcal{S}^{*}$, the class of starlike functions in ${\Delta}$. It is obvious from \eqref{arg-def} that for each $0<r<1$, the curve $\Gamma_w$ is not allowed to have a loop which ensure the univalency of the function. But if for some $0\neq z\in \Delta$, $\Re(zf'(z)/f(z))<0$, then $f$ is not starlike with respect to $0$, or equivalently we can say that the image curve $\Gamma_w: f(|z|=r)$ is definitely not starlike, but still it may or may not be univalent. From \eqref{charcter}, we also see the importance of the Carathe\'{o}dory functions by writing \eqref{charcter} in terms of subordination as:
\begin{equation}\label{star-subord}
\frac{zf'(z)}{f(z)}\prec \frac{1+z}{1-z} \quad (z\in\Delta),
\end{equation}
where the symbol $\prec$ stands for the usual subordination.
In 1992, Ma-Minda~\cite{minda94} generalized \eqref{star-subord} by unifying all the subclasses of starlike functions as follows:
\begin{equation}\label{mindaclass}
\mathcal{S}^*(\Psi):= \biggl\{f\in \mathcal{A} : \frac{zf'(z)}{f(z)} \prec \Psi(z) \biggl\},
\end{equation}
where $\Psi$ has positive real part, $\Psi({\Delta})$ symmetric about the real axis with $\Psi'(0)>0$ and $\Psi(0)=1$. For some special classes, refer \cite{Goel-2020,Kanas-2000,PSharma-2019} and the references there in.

In view of the above, let us now consider the analytic univalent function $\psi$ in ${\Delta}$ such that $\psi(0)=0$, $\psi({\Delta})$ is starlike with respect to $0$ and introduce the following class of analytic functions:
\begin{equation}\label{gen-ma-min}
\mathcal{F}(\psi):= \left\{f\in \mathcal{A}: \frac{zf'(z)}{f(z)}-1 \prec \psi(z),\; \psi(0)=0 \right\}.
\end{equation}
Note that when $1+\psi(z)\not \prec (1+z)/(1-z)$, then the functions in the class $\mathcal{F}(\psi)$ may not be univalent in ${\Delta}$ which also implies $\mathcal{F}(\psi)\not\subseteq \mathcal{S}^{*}$. Thus in case, when the function $1+\psi:=\Psi$ has positive real part, $\Psi({\Delta})$ symmetric about the real axis with $\Psi'(0)>0$, then $\mathcal{F}(\psi)$ reduces to the class $\mathcal{S}^{*}(\Psi)$.
The functions in the class defined in \eqref{mindaclass} are univalent which help a lot in studying the geometrical properties of the functions, for example, Growth and Distortion theorems etc. But this may not be quite easy to study the analogous results in the class $\mathcal{F}(\psi)$.
In this direction, recently, Kargar et al.~\cite{kargar-2019} considered the following class, the first of it's kind: 
\begin{equation}\label{boothlem}
\mathcal{BS}(\alpha):= \biggl\{f\in \mathcal{A} : \frac{zf'(z)}{f(z)}-1 \prec \frac{z}{1-\alpha z^2},\; \alpha\in [0,1) \biggl\},
\end{equation}
where  $z/(1-\alpha z^2)=:\psi(z)$ (Booth lemniscate function~\cite{piejko-2013} and \cite{piejko-2015}) is an analytic univalent function and symmetric with respect to the real and imaginary axes. Note that the function $(1+z/(1-\alpha z^2))$ assumes negative values for $\alpha\neq0$, therefore functions in this class may not be univalent. For $f$ belonging to $\mathcal{BS}(\alpha)$, using the vertical strip domain $\{w\in \mathbb{C}: \mu <\Re{w}<\nu, \;\text{where}\; \mu<1<\nu \},$  Kargar et al.~\cite{kargar-2019} proved that $|f(z)/z|$ is bounded and obtained the coefficient estimates when $0\leq \alpha \leq 3-2\sqrt{2}$ along with Fekete-Szeg\"{o} inequality for the associated $k-th$ root transformation. In 2018, Najmadi et al.~\cite{NNEbadian-2018} obtained the bounds for the quantities $\Re{f(z)}$, $|f(z)|$ and $|f'(z)|$ when $0\leq \alpha\leq 3-2\sqrt{2}$. Recently, Kargar et al.~\cite{kar-Eba-2019} obtained the best dominant of the subordination $f(z)/z \prec F(z)$ for the range $0<\alpha\leq 3-2\sqrt{2}$ using the convolution technique, where $F(z)=\left({1+z\sqrt{\alpha}})/{(1-z\sqrt{\alpha}}\right)^{\frac{1}{2\sqrt{\alpha}}}.$ Cho et al.~\cite{cho-kumar-2018} dealt with the first order differential subordination implications and also solved the various sharp radius problems pertaining to the class $\mathcal{BS}(\alpha)$.

In 2019, Masih et al. \cite{Masih-2019} considered the following class with $\beta\in [0,1/2]$:
\begin{equation}\label{cissoidclass}
\mathcal{S}_{cs}(\beta):= \biggl\{f\in \mathcal{A} : \left(\frac{zf'(z)}{f(z)}-1\right) \prec \frac{z}{(1-z)(1+\beta z)},\; \beta\in [0,1) \biggl\}.
\end{equation}
They proved the growth theorem and also obtained the sharp estimates for the logarithmic coefficients but only for the range $\beta\in [0,1/2]$. Note that for $\beta\in[0,1/2]$, $\mathcal{S}_{cs}(\beta)$ is a Ma-Minda subclass, but for the other range, functions in this class may not be univalent.

In this paper, we establish the sharp growth theorem for the class $\mathcal{F}(\psi)$ with certain geometric conditions on $\psi$ and obtain covering theorem. Further provide some examples including newly defined classes are also discussed. As an application, we obtain growth theorem for the complete range of $\alpha$ and $\beta$ for the functions in the classes $\mathcal{BS}(\alpha)$ and $\mathcal{S}_{cs}(\beta)$, respectively that improves the earlier known bounds. Finally, the sharp Bohr-radii for the classes $S(\mathcal{BS}(\alpha))$ and $\mathcal{BS}(\alpha)$ are obtained. For some classes, we study the geometrical behavior of an analytic function of the form $f(z)/z$ which arises frequently while working with the class $\mathcal{S}^{*}(\Psi)$ and play an important role, for example, in obtaining the bounds for $\Re(f(z)/z)$ and $\arg(f(z)/z)$. The geometric properties and coefficients estimation   for the class $\mathcal{F}(\psi)$ are still open.

\section{Main Results}
Let $\mathcal{F}(\psi)$ be the class as defined in \eqref{gen-ma-min}. Now we begin with the following:
\begin{theorem}[Growth Theorem]\label{gen-thm1}
	If $\max_{|z|=r}\Re\psi(z)=\psi(r)$ and $\min_{|z|=r}\Re\psi(z)=\psi(-r)$. Then $f\in \mathcal{F}(\psi)$ satisfies the sharp inequalities:
	\begin{equation}\label{maingththm-eq}
	r \exp\left(\int_{0}^{r}\frac{\psi(-t)}{t}dt\right) \leq |f(z)| \leq
	r \exp\left(\int_{0}^{r}\frac{\psi(t)}{t}dt\right), \quad (|z|=r).
	\end{equation}
\end{theorem}
\begin{proof}
	Let $f\in \mathcal{F}(\psi)$. For $z=re^{i\theta}$, we have
	\begin{equation}\label{realbound}
	\phi(-r) \leq \Re{\psi(re^{i\theta})} \leq \phi(r).
	\end{equation}
	Let $\Phi(z)=\psi(\omega(z))$, where $\omega$ is a Schwarz function. Then from \eqref{mindaclass}, we have
	\begin{equation*}
	\log \frac{f(z)}{z}=\int_{0}^{z} \frac{\Phi(\zeta)}{\zeta}d\zeta.
	\end{equation*}
	Now by taking $\zeta=te^{i\beta}$ so that $d\zeta=e^{i\beta}dt$, where $\beta$ is fixed but arbitrary and $z=re^{i\beta}$, we have
	\begin{equation}\label{log-func}
	\log \frac{f(z)}{z}=\int_{0}^{r}\frac{\Phi(te^{i\beta})}{t}dt.
	\end{equation}	
	From the Maximum-minimum modulus principle, we find that $\Phi$ also satisfies the inequality \eqref{realbound}. Therefore, without loss of generality, we may replace $\Phi$ by $\psi$ and $\beta$ by $\theta$ in \eqref{log-func}. Then by equating real parts on either side of \eqref{log-func}, we have
	\begin{equation}\label{main-eq}
	\log\left|\frac{f(z)}{z}\right| =\int_{0}^{r} \frac{\Re{\Phi(te^{i\theta})}}{t}dt
	\end{equation}
	and now using the inequalities \eqref{realbound} in \eqref{main-eq}, we obtain
	\begin{equation*}
	\int_{0}^{r} \frac{\psi(-t)}{t}dt \leq \log \left|\frac{f(z)}{z}\right| \leq \int_{0}^{r} \frac{\psi(t)}{t}dt,
	\end{equation*}
	and \eqref{maingththm-eq} follows. The result is sharp for the function
	\begin{equation}\label{f0}
	f_0(z)=z \exp\int_{0}^{z}\frac{\psi(t)}{t}dt.
	\end{equation}
	This completes the proof. 
\end{proof}

\begin{remark}
	In the above theorem, we chose $\max_{|z|=r}\Re\psi(z)=\psi(r)$ and $\min_{|z|=r}\Re\psi(z)=\psi(-r)$ for computational convenience. However, these conditions may change according to the choice of $\psi$ in that case, appropriately these may be replaced. 
\end{remark}
\begin{remark}
	If $1+\psi$ is a Carath\'{e}odory univalent function, then Theorem~\ref{gen-thm1} reduces to the result~\cite[Corollary~1, p.~161]{minda94} and	moreover, letting $r$ tends to $1$ in Theorem~\ref{gen-thm1}, we obtain the covering theorem (Koebe-radius) for the class $\mathcal{F}(\psi)$.
\end{remark}

\begin{corollary}[Covering Theorem]
	If  $f\in \mathcal{F}(\psi)$ and $f_0$ as defined in \eqref{f0}, then either $f$ is a rotation of $f_0$ or
	$$	\{w\in{\Delta} : |w|\leq-{f}_0(-1) \} \subset f({\Delta}),$$
	where $-{f}_0(-1)=\lim_{r\rightarrow 1}(-f_0(-r)).$
\end{corollary}

Let $L(f,r)$ denotes the length of the boundary curve  $f(|z|=r)$. Note that for $z=re^{i\theta}$, we have $L(f,r):=\int_{0}^{2\pi}|zf'(z)|d\theta$. Now we obtain the following result:
\begin{corollary}
	Assume that $\max_{|z|=r}|\psi(z)|=\psi(r)$ and also $\psi$ satisfies the conditions of Theorem~\ref{gen-thm1}. Let $M(r)=\exp\left(\int_{0}^{r}\frac{\psi(t)}{t}dt\right)$.  If $f\in \mathcal{F}(\psi)$, then for $|z|=r$, we have
	\begin{equation*}
	\Re\frac{f(z)}{z} \leq M(r), \;\;
	|f'(z)| \leq (1+\psi(r)) M(r)\;\;
	\text{and}\;\;
	L(f,r) \leq 2\pi r (1+\psi(r)) M(r) .
	\end{equation*}
\end{corollary}	

Let $$\psi(z)= \left\{
\begin{array}{ll}
\beta z/(1+\alpha z), & \hbox{$\beta>0$, $0< \alpha <1$ ;} \\
\eta z, & \hbox{$\eta>0$.}
\end{array}
\right.$$
Then the above two choices of $\psi$ are clearly convex univalent and $\psi({\Delta})$ are symmetric about real axis as $\overline{\psi(z)}=\psi(\bar{z})$. It is further evident that $ 1+\psi(z)\not\prec (1+z)/(1-z)$ except for the second choice of $\psi$ when $0<\eta\leq1$.
We now obtain the following sharp result from Theorem~\ref{gen-thm1}:
\begin{example}
	Let $f\in \mathcal{F}({\beta z}/{(1+\alpha z)})$, where $\beta>0$ and $0< \alpha <1$ and $|z|=r$. Then
	
	$$r(1-\alpha r)^{\frac{\beta}{\alpha}} \leq |f(z)|\leq r(1+\alpha r)^{\frac{\beta}{\alpha}},$$ which implies:
	$$\left\{w : |w|\leq (1-\alpha)^{\frac{\beta}{\alpha}} \right\} \subset f({\Delta}),\;\;
	|f'(z)|\leq \left(1+\frac{\beta r}{1+\alpha r}\right)(1+\alpha r)^{\frac{\beta}{\alpha}}\;\; \text{and}\;\;
	\Re\frac{f(z)}{z} \leq (1+\alpha r)^{\frac{\beta}{\alpha}}.$$
\end{example}	

\begin{example}
	Let $f\in \mathcal{F}(\eta z)$, where $\eta>0$ and $|z|=r$. Then
	$$r \exp(-\eta r ) \leq |f(z)|\leq r \exp(\eta r),$$ which implies:
	$$\left\{w : |w|\leq \exp({-\eta}) \right \} \subset f({\Delta}),\;\;
	|f'(z)|\leq (1+\eta r) \exp(\eta r)\;\; \text{and}\;\;
	\Re\frac{f(z)}{z}\leq \exp(\eta r).$$
\end{example}

From the above examples, it is clear that $f\in \mathcal{F}(\psi)$ if and only if 
\begin{equation*}
\frac{zf'(z)}{f(z)}\in
\left\{
\begin{array}
{lr}
\Omega_1,    & \text{when}\;  \psi(z)={\beta z}/{(1+\alpha z)}; \\
\Omega_2,   &\text{when}\; \psi(z)=\eta z,
\end{array}
\right.
\end{equation*}
where  $\Omega_1=\{w\in \mathbb{C}: |w-1|< |\beta-\alpha(w-1)|\}$ and $\Omega_2=\{w\in \mathbb{C}: |w-1|< \eta\}$, respectively for $z\in{\Delta}$.
\section{Some Applications and Further results}\label{sec-1}

\subsection{On Booth-Lemniscate} 
Let $\mathcal{BS}(\alpha)$ be the class as defined in \eqref{boothlem}.
\begin{theorem}\label{grth}
	Let $0< \alpha<1$ and $f\in \mathcal{BS}(\alpha)$, then for $|z|=r$
	\begin{equation}\label{grth-thm}
	-\hat{f}(-r)\leq|f(z)|\leq \hat{f}(r),
	\end{equation}
	where
	\begin{equation}\label{hat}
	\hat{f}(z)=z\left(\frac{1+z\sqrt{\alpha}}{1-z\sqrt{\alpha}}\right)^{\frac{1}{2\sqrt{\alpha}}}.
	\end{equation}
	The result is sharp.
\end{theorem}
\begin{proof}
	Let $\psi(z):= {z}/{(1-\alpha z^2)}$ and $f\in \mathcal{BS}(\alpha):=\mathcal{F}(\psi)$. For $z=re^{i\theta}$, we have
	\begin{equation*}
	- \frac{r}{1-\alpha r^2} \leq \Re{\psi(re^{i\theta})} \leq \frac{r}{1-\alpha r^2}
	\end{equation*}
	and
	\begin{equation*}
	-\int_{0}^{r} \frac{1}{1-\alpha t^2}dt \leq \log \left|\frac{f(z)}{z}\right| \leq \int_{0}^{r} \frac{1}{1-\alpha t^2}dt,
	\end{equation*}
	where
	$$\int_{0}^{r} \frac{1}{1-\alpha t^2}dt=\frac{1}{2\sqrt{\alpha}} \log{\frac{1+\sqrt{\alpha}r}{1-\sqrt{\alpha}r}}.$$
	Hence, the result follows from Theorem~\ref{gen-ma-min}. 
\end{proof}

\begin{remark}
	Theorem~\ref{grth} improves the upper bound of $\Re{f(z)}$ and bounds of $|f(z)|$, obtained in \cite[Theorem~2, p.~116]{NNEbadian-2018} and \cite[Theorem~3, p.~116]{NNEbadian-2018} respectively.
\end{remark}

We now extend \cite[Theorem~2.6, p.~1238]{kar-Eba-2019} for the complete range of $\alpha$ using Theorem~\ref{grth}:
\begin{corollary}
	Let $f\in\mathcal{BS}(\alpha)$, $\alpha\in(0,1)$ and $|z|=r$, then
	$$
	\Re\frac{f(z)}{z} \leq \left(\frac{1+r\sqrt{\alpha}}{1-r\sqrt{\alpha}}\right)^{\frac{1}{2\sqrt{\alpha}}}
	\;\text{and} \quad
	|f'(z)|\leq \left(1+\frac{r}{1-\alpha r^2}\right) \left(\frac{1+r\sqrt{\alpha}}{1-r\sqrt{\alpha}}\right)^{\frac{1}{2\sqrt{\alpha}}}.
	$$	
	The result is sharp for the function $\hat{f}$ given in \eqref{hat}.
\end{corollary}	

\begin{corollary}
	Let $\alpha\in (0,1)$ be fixed. Then $f\in \mathcal{BS}(\alpha)$ satisfies the inequality
	\begin{equation*}
	L(f,r) \leq 2\pi r\left(1+\frac{r}{1-\alpha r^2}\right) \left(\frac{1+r\sqrt{\alpha}}{1-r\sqrt{\alpha}}\right)^{\frac{1}{2\sqrt{\alpha}}},\quad (|z|=r).
	\end{equation*}
\end{corollary}	

\begin{corollary}[Koebe-radius]\label{r*}
	Let $0< \alpha<1$ and $\hat{f}$ as given in \eqref{hat}. If  $f\in \mathcal{BS}(\alpha)$, then either $f$is a rotation of $\hat{f}$ or
	$$	\{w\in \mathbb{C} : |w|\leq-\hat{f}(-1) \} \subset f({\Delta}).$$
\end{corollary}
\begin{proof}
	The proof follows by letting $r$ tends to $1$ in the inequality $-\hat{f}(-r)\leq|f(z)|$, given in  \eqref{grth-thm}.
\end{proof}	

\begin{theorem}
	Let $\alpha\in (0,2-\sqrt{3}]$ be fixed. Then $f\in \mathcal{BS}(\alpha)$ satisfies the sharp inequality
	\begin{equation*}
	\left|\arg\frac{f(z)}{z}\right| \leq \max_{|z|=r}\; \arg\left(\frac{1+z\sqrt{\alpha}}{1-z\sqrt{\alpha}}\right)^{\frac{1}{2\sqrt{\alpha}}}.
	\end{equation*}
\end{theorem}	
\begin{proof}
	From \cite[Theorem~2.5, p.~1238]{kar-Eba-2019}, we have $f(z)/z \prec \hat{f}(z)/z$ for $0<\alpha\leq 2-\sqrt{3}$, where $\hat{f}$ is defined in \eqref{hat}. Since the function $\hat{f}(z)/z$ is convex and symmetric about the real axis in ${\Delta}$, therefore we easily see that
	$$\left(\frac{1-\sqrt{\alpha}}{1+\sqrt{\alpha}}\right)^{\frac{1}{2\sqrt{\alpha}}}>0.$$ 
	Thus $\hat{f}(z)/z$ is a Carathe\'{o}dory function and the result follows.	
\end{proof}	

For our next result, we need the following definition and a related class: 
\begin{definition}
	Let $f(z)=\sum_{k=0}^{\infty}a_kz^k$ and $g(z)=\sum_{k=0}^{\infty}b_kz^k$ are analytic in ${\Delta}$ and $f({\Delta})=\Omega$. Consider a class of analytic functions $S(f):=\{g : g\prec f\}$ or equivalently $S(\Omega):=\{g : g(z)\in \Omega\}$. Then the class $S(f)$ is said to satisfy Bohr-phenomenon, if there exists a constant $r_0\in (0,1]$ such that the inequality
	$\sum_{k=1}^{\infty}|b_k|r^k \leq d(f(0),\partial\Omega)$
	holds for all $|z|=r\leq r_0$, where $d(f(0),\partial\Omega)$ denotes the Euclidean distance between $f(0)$ and the boundary of $\Omega=f({\Delta})$.
	The largest such $r_0$ for which the inequality holds, is called the Bohr-radius.
\end{definition}
See the articles \cite{jain2019,bhowmik2018} and the references therein for more. Let us now introduce the following class:
\begin{equation*}\label{bohrclass}
S(\mathcal{BS}(\alpha)):= \biggl\{g : g\prec f, \; g(z)=\sum_{k=1}^{\infty}b_k z^k \;\text{and}\; f\in \mathcal{BS}(\alpha)  \biggl \}.
\end{equation*}

\begin{theorem}[Booth-Bohr-radius]
	The class $S(\mathcal{BS}(\alpha))$ satisfies Bohr-phenomenon in $|z|\leq r(\alpha)$, where $r(\alpha)$ is the unique positive root of the equation
	\begin{equation}\label{boothbohr-eq}
	r\left(\frac{1+r\sqrt{\alpha}}{1-r\sqrt{\alpha}}\right)^{\frac{1}{2\sqrt{\alpha}}}- \left(\frac{1-\sqrt{\alpha}}{1+\sqrt{\alpha}}\right)^{\frac{1}{2\sqrt{\alpha}}}=0,
	\end{equation}
	whenever $0<\alpha \leq 3-2\sqrt{2}$. The result is sharp for the function $\hat{f}$ given in \eqref{hat}.
\end{theorem}
\begin{proof}
	Since  $g\in S(\mathcal{BS}(\alpha))$, we have $g\prec f$ for a fixed $f\in \mathcal{BS}(\alpha)$. From Corollary~\ref{r*}, we obtain the Koebe-radius $r_{*}=-\hat{f}(-1)$ such that $  r_{*}\leq d(0,\partial\Omega)=|f(z)|$ for $|z|=1$. Also using \cite[Theorem~2.5, p.~1238]{kar-Eba-2019}, we have
	\begin{equation}\label{f-f0}
	\frac{f(z)}{z}\prec \frac{\hat{f}(z)}{z}.
	\end{equation}
	Recall the result \cite[Lemma~1, p.1090]{bhowmik2018}, which reads as:
	let $f$ and $g$ be analytic in ${\Delta}$ with $g\prec f,$ where
	$f(z)=\sum_{n=0}^{\infty}a_n z^n$ and $ g(z)= \sum_{k=0}^{\infty}b_k z^k.$
	Then
	$\sum_{k=0}^{\infty}|b_k|r^k \leq \sum_{n=0}^{\infty}|a_n|r^n$ for $ |z|=r\leq1/3.$
	Now using the result for  $g\prec f$ and  \eqref{f-f0}, we have
	\begin{equation*}
	\sum_{k=1}^{\infty}|b_k|r^k \leq	r+\sum_{n=2}^{\infty}|a_n|r^n \leq \hat{f}(r)\quad\text{for}\; |z|=r\leq1/3.
	\end{equation*}
	Finally, to establish the inequality
	$\sum_{k=1}^{\infty}|b_k|r^k \leq d(f(0),\partial\Omega),$
	it is enough to show $\hat{f}(r) \leq r_{*}.$
	But this holds whenever $r\leq r(\alpha)$, where $r(\alpha)$ is the least positive root of the equation $\hat{f}(r)=r_{*}.$ Now let $T(r):=\hat{f}(r)-r_{*}$, then
	\begin{equation*}
	T'(r)=\left(\frac{1+r\sqrt{\alpha}}{1-r\sqrt{\alpha}}\right)^{\frac{1}{2\sqrt{\alpha}}}+ r\left(\frac{1+r\sqrt{\alpha}}{1-r\sqrt{\alpha}}\right)^{\frac{1}{2\sqrt{\alpha}}-1}\frac{1}{(1-r\sqrt{\alpha})^2}.
	\end{equation*}
	Since $(1+r\sqrt{\alpha})/(1-r\sqrt{\alpha})>0$, therefore $T'(r)>0$ and so $T$ is an increasing function of $r$. Also $T(0)<0$ and $T(1)>0$. Thus
	the existence of the root $r(\alpha)$ is ensured by the Intermediate Value theorem for the continuous functions. By a computation, it can easily be seen that $r(\alpha)< 1/3$ and
	hence the result. 
\end{proof}

\begin{corollary}
	Let $0<\alpha \leq 3-2\sqrt{2}$. The Bohr-radius for the class $\mathcal{BS}(\alpha)$ is $r(\alpha)$, where $r(\alpha)$ is the unique positive root of the Eq.~\eqref{boothbohr-eq}.
\end{corollary}

\subsection{On Cissoid of Diocles}
Let us consider
\begin{equation*}
S_{\beta}(z)= \frac{z}{(1-z)(1+\beta z)}=\frac{1}{1+\beta}\left(\frac{1}{1-z}+\frac{1}{1+\beta z}\right)= \sum_{n=1}^{\infty}\frac{1-(-\beta)^n}{1+\beta}z^n,
\end{equation*}
where $\beta\in [0,1)$. Clearly, it is analytic, symmetric about the real-axis and 
maps the unit disk ${\Delta}$ onto the domain bounded by {\it Cissoid of Diocles}:
\begin{equation*}
CS(\beta):=\left\{ w=u+iv\in \mathbb{C}: \left(u-\frac{1}{2(\beta-1)}\right)(u^2+v^2)+ \frac{2\beta}{(1+\beta)^2(\beta-1)}v^2=0 \right\}.
\end{equation*}
Let us now consider the class $\mathcal{S}_{cs}(\beta)$ as defined in \eqref{cissoidclass}. Masih et al. \cite{Masih-2019} considered this class with $\beta\in [0,1/2]$ since $\Re(1+{z}/{((1-z)(1+\beta z))}\geq (2\beta-1)/(2(\beta-1))\geq0$. Clearly,  $\mathcal{S}_{cs}(\beta) = \mathcal{F}(S_{\beta}(z))$ for $\beta\in [0,1)$ and we have the following result:
\begin{theorem}\label{cissoidgth}
	Let $f\in \mathcal{S}_{cs}(\beta)$ and $\beta\in [0,1)$. Then
	\begin{equation*}
	-\tilde{f}(-r) \leq |f(z)| \leq \tilde{f}(r),
	\end{equation*}
	where 
	\begin{equation}\label{tilde-f}
	\tilde{f}(z)=z\left(\frac{1+\beta z}{1-z}\right)^{\frac{1}{1+\beta}}.
	\end{equation}
\end{theorem}
\begin{proof}
	Let $\psi(z):={z}/{((1-z)(1+\beta z))}$ and $f\in \mathcal{S}^{*}_{cs}(\beta):=\mathcal{F}(\psi)$. Following the proof of \cite[Theorem~3.1, p.~5]{Masih-2019}, it is easy to see that for $z=re^{i\theta}$, where $\theta\in[0,2\pi]$, we have
	\begin{equation*}
	\min_{|z|=r}\Re\psi(z)= \frac{-r+(\beta-1)r^2+\beta r^3}{(1+r)^2(1-\beta r)^2}=\psi(-r)
	\end{equation*}
	and 
	\begin{align*}
	\max_{|z|=r}\Re\psi(z) &= \lim_{\theta\rightarrow 0}\frac{-r^2+\beta r^2 -\beta r^3 \cos{\theta}+r \cos{\theta}}{(1+r^2-2r \cos{\theta})(1+{\beta}^2 r^2+2\beta r \cos{\theta})}\\
	&\leq \frac{\beta-1}{2(1+\beta)^2}=\max_{|z|=1}\Re\psi(z).
	\end{align*} 
	Thus, we have $\psi(-r)\leq\Re\psi(z)\leq \psi(r)$ for $r\neq1$ and $1/(2(\beta-1))=\psi(-1)\leq \Re\psi(z)\leq (\beta-1)/(2(\beta+1)^2)$ for $r=1$. 
	Also note that 
	\begin{equation*}
	\tilde{f}(z)=z\exp\int_{0}^{z}\frac{\psi(t)}{t}dt=z\left(\frac{1+\beta z}{1-z}\right)^{\frac{1}{1+\beta}}. 
	\end{equation*}
	Now the result follows from Theorem~\ref{gen-ma-min}. 
\end{proof}
\begin{remark}\label{3.2}
	Let $\tilde{F}(z)={\tilde{f}(z)}/{z}$ and $|z|=1$ , where $\tilde{f}$ is as defined in Theorem~\ref{cissoidgth}. A calculation show that
	\begin{equation*}
	1+\frac{\tilde{F}''(z)}{\tilde{F}'(z)}= 1+ \frac{-\beta z}{(1+\beta z)(1-z)}+\frac{2z}{1-z},
	\end{equation*}
	which implies that
	\begin{equation*}
	\Re\left(1+\frac{\tilde{F}''(z)}{\tilde{F}'(z)}\right) \geq \beta \Re\left(\frac{-z}{(1+\beta z)(1-z)}\right).
	\end{equation*}
	Since $$\Re\left(\frac{-z}{(1+\beta z)(1-z)}\right)=\frac{1-\beta}{2(1-{\beta}^2+2\beta \cos{\theta})}=:g(\theta),$$
	and a simple calculation shows that $g$ attains its minimum at $\theta=0$. Therefore, we have
	$$ 1+\frac{\tilde{F}''(z)}{\tilde{F}'(z)}\geq \frac{\beta(1-\beta)}{2(1+\beta)^2}\geq0.$$ 
	Hence $\tilde{F}$ is convex univalent in ${\Delta}$. 
\end{remark}
Observe that the function $	S_{\beta}(z)$ is not convex when $\beta\neq 0$ and the result, $f(z)/z\prec \tilde{F}(z)$ similar to theorem~\ref{fk-z} is still open for $f\in \mathcal{S}_{cs}(\beta)$. 
By letting $r$ tends to $1$ in the above Theorem~\ref{cissoidgth}, we obtain:
\begin{corollary}[Koebe-radius]\label{cissoidkoebe}
	Let $\tilde{f}$ as given in \eqref{tilde-f}. If  $f\in \mathcal{S}_{cs}(\beta)$, then either $f$is a rotation of $\tilde{f}$ or
	$$	\left\{w\in \mathbb{C} : |w|\leq-\tilde{f}(-1)=\left(\frac{1-\beta}{2}\right)^{1/(1+\beta)} \right\} \subset f({\Delta}).$$
\end{corollary}		
\begin{remark}
	We improved the result \cite[Corollary~4.3.1, p.~8]{Masih-2019} in Theorem~\ref{cissoidgth} and Corollary~\ref{cissoidkoebe} by extending the range of $\beta$.
\end{remark}

\subsection{Modified Koebe function:}
The Koebe function $k(z)=z/(1-z)^2$ has a pole at $z=1$ and maps unit disk onto the domain $\mathbb{C}-(-\infty, 1/4]$, which is a slit domain. We now introduced the modified Koebe function: 
\begin{equation}\label{modkoebe}
K(z):= \frac{z}{(1+\eta z)^2}, \quad 0\leq \eta <1,
\end{equation}
which is bounded in ${\Delta}$ and symmetric about the real-axis. It is interesting to observe the geometry of the domain $K({\Delta})$, which assumes different shapes for different choices of $\eta$ such as a convex or a Bean or a Cardioid shaped domain. Especially when $\eta$ tends to $1$, we see that one of the rotation of the image domain $K(\Delta)$ will converge to $k({\Delta})$ and thereby justifying the name of $K(z)$. Since $k(z)= (u^2(z)-1)/4$, where $u(z)=(1+z)/(1-z)$, in a similar fashion, we can write 
\begin{equation*}
K(z)= \frac{1}{4\eta}(1-v^2(z)),
\end{equation*} 
where 
$ v(z)={(1-\eta z)}/{(1+\eta z)}$ and $\eta\neq0.$
\begin{lemma}\label{modkoebeconvex}
	The function $K(z)$ as defined in \eqref{modkoebe} is convex for $0\leq \eta\leq 2-\sqrt{3}$. 
\end{lemma}
\begin{proof}
	Let $K(z)=z/ (1+\eta z)^2$. When $\eta=0$, $K(z)$ is the identity function and hence is convex. So let us consider $0<\eta<1$. By a computation, we obtain that
	\begin{equation*}
	1+\frac{zK''(z)}{K'(z)}= \frac{1-4\eta z+ \eta^2 z^2}{(1-\eta z)(1+\eta z)}.
	\end{equation*}
	Putting $z=e^{i \theta}$, we have
	\begin{equation}\label{realexpress}
	\Re\left(1+\frac{zK''(z)}{K'(z)}\right)=\frac{1-4 \eta(1-\eta^2)\cos{\theta}-\eta^4(1+\cos{\theta}){\sin}^2{\theta} }{((1+\eta^2)^2-(2\eta \cos{\theta})^2)}.
	\end{equation}
	Since $((1+\eta^2)^2-(2\eta \cos{\theta})^2)>0$ for all $\theta$ and for each fixed $\eta$. Therefore, we now only need to consider the numerator in \eqref{realexpress}.
	A computation reveals that 
	\begin{equation*}
	N(\theta):= 1-4 \eta(1-\eta^2)\cos{\theta}-\eta^4(1+\cos{\theta}){\sin}^2{\theta}
	\end{equation*}
	is increasing in $0\leq \theta \leq \pi$ (note that $N(\theta)=N(-\theta)$) with $N(\theta)\geq 0$ when $0<\eta\leq 2-\sqrt{3}$, while $N(\theta)$ takes negative values when $\eta>2-\sqrt{3}$. Hence by  the definition of convexity, result follows.   
\end{proof}

Now let us consider the function 
\begin{equation*}
\psi(z):= \frac{\gamma z}{(1+ \eta z)^2 }= \gamma K(z), \quad \text{where} \quad \gamma>0,
\end{equation*}	
and introduce a related class defined as follows:
\begin{equation}
\mathcal{S}_{\gamma}(\eta):= \biggl\{f\in \mathcal{A} : \left(\frac{zf'(z)}{f(z)}-1\right) \prec \frac{\gamma z}{(1+\eta z)^2},\; \eta\in [0,1),\; \gamma>0 \biggl\}. 
\end{equation}

Note that if $\gamma$ and $\eta$ satisfies the condition $(1-\eta)^2\geq\gamma$, then the class $\mathcal{S}_{\gamma}(\eta)$ reduces to a Ma-Minda subclass of univalent starlike functions.
Also letting $\eta=1/4$ and $\gamma=25(\sqrt{2}-1)/16$, we see that the class $\mathcal{S}^{*}(\sqrt{1+z}) \subset \mathcal{S}_{\gamma}(\eta)$. 

\begin{theorem}\label{modkobeGrth}
	Let $f\in \mathcal{S}_{\gamma}(\eta)$ and $\eta\in [0,2-\sqrt{3}]$. Then
	\begin{equation*}
	-\kappa(-r) \leq |f(z)| \leq \kappa(r),
	\end{equation*}
	where 
	\begin{equation*}\label{kappa-f}
	{\kappa}(z):=z \exp\left(\frac{\gamma z}{(1+\eta z)^2}\right).
	\end{equation*}	
\end{theorem}
\begin{proof}
	Since $\psi(z)=\gamma K(z)$, using Lemma~\ref{modkoebeconvex}, we see that for $|z|=r$, $\psi(-r)\leq \Re \psi(z)\leq \psi(r).$
	Also, we have $\kappa(z)=z\exp\int_{0}^{z}({\psi(t)}/{t})dt$. Hence, the result follows from Theorem~\ref{gen-ma-min}.
\end{proof}	

Using Lemma~\ref{modkoebeconvex}, we also obtain that $\Re \psi(z)\geq \psi(-r)$ for all $\eta\in [0,1)$ which implies $-\kappa(-r) \leq |f(z)|$. So we have the following results:
\begin{corollary}[Radius of starlikeness]
	Let $f\in \mathcal{S}_{\gamma}(\eta)$, $\gamma>0$ and $\eta\in [0,1)$. Then $f$ is starlike (univalent) of order $\alpha\in [0,1)$ inside the disk $|z|<r_0$, where $r_0$ is the smallest positive root of the equation
	$$(1-\alpha)\eta^2 r^2-(2(1-\alpha)\eta+\gamma)r+(1-\alpha)=0.$$
	
\end{corollary}
\begin{corollary}[Koebe-radius]
	Let $f\in \mathcal{S}_{\gamma}(\eta)$ and $\eta\in [0,1)$. Then either $f$ is a rotation of $\kappa$ or 
	\begin{equation*}
	\left\{w\in \mathbb{C} : |w|\leq-\kappa(-1)= \exp\left(\frac{-\gamma }{(1-\eta)^2}\right) \right\} \subset f({\Delta}).
	\end{equation*}
\end{corollary}
\begin{remark}\label{Fk-convex}
	Let $F_{\kappa}(z):= \kappa(z)/z= \exp(\gamma z/(1+\eta z)^2)$. We see that for $\eta=0$ and $\gamma\leq1$, $F_{\kappa}$ is clearly convex. So consider $0<\eta<1$. After some calculations, we obtain that
	\begin{equation*}
	G(z):=1+ \frac{zF''_{\kappa}(z)}{F'_{\kappa}(z)} =\frac{\eta^4 z^4 +(2\eta^3+\gamma\eta^2)z^3-(6\eta^2+2\eta\gamma)z^2+ (\gamma-2\eta)z+1}{(1+\eta z)^3(1-\eta z)}.
	\end{equation*}
	Now for $z=e^{i\theta}$, the denominator of the real part of $G$ is $(1+\eta^2-2\eta \cos{\theta})(1+\eta^2+2\eta \cos{\theta})^3>0$, since $(1-\eta)^2>0$ and therefore, it suffices to consider the numerator. After a rigorous computation, we find that numerator of the real part of $G$ is non negative if and only if
	$0<\gamma<1$ and $0<\eta\leq \eta_0$, where $\eta_0$ (depends on $\gamma$) is the smallest positive root of the equation
	\begin{equation}\label{eta0-def}
	(1-\gamma)+(3\gamma-10) \eta^2+12 \eta^3+(8-3\gamma)\eta^4-16\eta^5+(2+\gamma)\eta^6+4\eta^7-\eta^8=0.
	\end{equation}  
	Hence, $F_{\kappa}$ convex for $0<\gamma<1$ and $0<\eta\leq \eta_0$.
\end{remark}

For our next result, we need to recall the following result of  Ruscheweyh and Stankiewicz~\cite{RusStan-1985}:
\begin{lemma}[\cite{RusStan-1985}]\label{convo-result}
	Let the analytic functions $F$ and $G$ be convex univalent in ${\Delta}$. If $f\prec F$ and $g \prec G$, then $$f*g \prec F*G \quad (z\in{\Delta}).$$
\end{lemma}
\begin{theorem}\label{fk-z}
	Let $\eta\in [0,2-\sqrt{3}]$. If $f$ belongs to the class $\mathcal{S}_{\gamma}(\eta) $, then 
	$$\frac{f(z)}{z} \prec F_{\kappa}(z),\quad (z\in {\Delta})$$
	where $F_{\kappa}(z)=\kappa(z)/z$ is the best dominant and $\kappa$ as defined in Theorem~\ref{modkobeGrth}.
\end{theorem}
\begin{proof}
	Let $f\in \mathcal{S}_{\gamma}(\eta)$, then 
	\begin{equation}\label{phi-psi}
	\phi(z):= \frac{zf'(z)}{f(z)}-1 \prec \psi(z).
	\end{equation}
	It is well-known that the function
	\begin{equation*}
	g(z)=\log \left(\frac{1}{1-z}\right) =\sum_{n=1}^{\infty}\frac{z^n}{n} \in \mathcal{C},
	\end{equation*}
	where $\mathcal{C}$ is the usual class of normalized convex(univalent) function and thus for  $f\in \mathcal{A}$, we get
	\begin{equation}\label{phi-gExp1}
	\phi(z)* g(z)= \int_{0}^{z}\frac{\phi(t)}{t}dt \quad \text{and}\quad \psi(z)* g(z)= \int_{0}^{z}\frac{\psi(t)}{t}dt. 
	\end{equation}
	From Lemma~\ref{modkoebeconvex}, we see that $\psi$ is convex for $\eta\in [0,2-\sqrt{3}]$. Thus applying Lemma~\ref{convo-result} in \eqref{phi-psi}, we get
	\begin{equation}\label{phi-gExp2}
	\phi(z)* g(z) \prec \psi(z)*g(z).
	\end{equation}
	Now from \eqref{phi-gExp1} and \eqref{phi-gExp2}, we obtain
	\begin{equation*}
	\int_{0}^{z}\frac{\phi(t)}{t}dt \prec \int_{0}^{z}\frac{\psi(t)}{t}dt, 
	\end{equation*}
	which  implies that
	\begin{equation*}
	\frac{f(z)}{z}:= \exp\int_{0}^{z}\frac{\phi(t)}{t}dt \prec \exp\int_{0}^{z}\frac{\psi(t)}{t}dt=: \frac{\kappa(z)}{z}.
	\end{equation*}
	This completes the proof. 
\end{proof}

\begin{corollary}
	Let $0<\gamma<1$ and $0<\eta\leq \min\{2-\sqrt{3},\eta_0\}$, where $\eta_0$ is the least positive root of the equation \eqref{eta0-def} and also let $0<\gamma\leq \pi/2$ when $\eta=0$. If $f\in \mathcal{S}_{\gamma}(\eta)$, then $f$ satisfies the sharp inequality
	$$\left|\arg\frac{f(z)}{z}\right|\leq \max_{|z|=r}\; \arg\exp\left(\frac{\gamma z}{(1+\eta z)^2}\right).$$
\end{corollary}
\begin{proof}
	Let  $F_{\kappa}(z):=\kappa(z)/z=\exp(\gamma z/(1+\eta z)^2)$ which is symmetric about the real axis. From Theorem~\ref{fk-z}, have $f(z)/z \prec F_{\kappa}(z)$ for $0\leq \eta\leq 2-\sqrt{3}$. Since for $\eta=0$, $\Re F_{\kappa}(z)>0$ if and only $\gamma\leq \pi/2$. The result is obvious. Now from  Remark~\ref{Fk-convex}, we see that  if $0<\gamma<1$ and $0<\eta\leq \min\{2-\sqrt{3},\eta_0\}$, where $\eta_0$ is the least positive root of the equation \eqref{eta0-def} then $F_{\kappa}$ is convex which implies
	$$\Re F_{\kappa}(z)\geq \exp\left(\frac{-\gamma}{(1-\eta)^2}\right)>0,$$
	and $F_{\kappa}$ is also a Carathe\'{o}dory function in this case. Hence the result follows.
\end{proof}

Now using Theorem~\ref{modkobeGrth}, Remark~\ref{Fk-convex} and Theorem~\ref{fk-z}, we obtain the following result:
\begin{theorem}
	Let $f\in \mathcal{S}_{\gamma}(\eta)$, then
	$$\Re\left(\frac{f(z)}{z}\right) \leq \exp\left(\frac{\gamma r}{(1+\eta r)^2}\right) \quad \text{for}\quad \eta\in[0,1)$$ 
	and 
	$$\min_{|z|=r} \exp\left(\frac{\gamma z}{(1+\eta z)^2}\right) \leq \Re\left(\frac{f(z)}{z}\right) \quad \text{for}\quad \eta\in[0,2-\sqrt{3}]. $$
	In partcular, if $0<\gamma<1$ and $0<\eta\leq \min\{2-\sqrt{3},\eta_0\}$, where $\eta_0$ is the least positive root of the equation \eqref{eta0-def}, then 
	$$ \exp\left(\frac{-\gamma r }{(1-\eta r)^2}\right) \leq \Re\left(\frac{f(z)}{z}\right).$$
	The result is sharp.	
\end{theorem}

We conclude this paper by introducing the following three new subclasses of $\mathcal{F}(\psi)$:
\begin{equation*}
\mathcal{T}:= \biggl\{f\in \mathcal{A} : \frac{zf'(z)}{f(z)}-1 \prec \log(1-z) \biggl\},
\end{equation*}
which means $zf'(z)/f(z) \in \{w\in \mathbb{C}: |\exp(w-1)-1|<1 \},$
\begin{equation*}
\mathcal{S}_{p}:= \biggl\{f\in \mathcal{A} : \frac{zf'(z)}{f(z)} \prec 1- \left(\log\frac{1+\sqrt{z}}{1-\sqrt{z}}\right)^2  \biggl\},
\end{equation*}
or equivalently $zf'(z)/f(z)\in \{w \in \mathbb{C}:|1-w|<\Re((1-w)+{\pi}^2) \}$, a parabola with opening in left half plane
and
\begin{equation*}
\mathcal{L}(\beta):= \biggl\{f\in \mathcal{A} : \frac{zf'(z)}{f(z)}-1 \prec \frac{z}{\cos(\beta z)},\; \beta\in [0,1] \biggl\}.
\end{equation*}
The above new classes are still open to study. Also see figure~\ref{f1}. Note that for the classes $\mathcal{T}$ and $\mathcal{L}(\beta)$, the function $f_0$ defined in \eqref{f0} takes the respective particular form
\begin{equation*}
f_{\mathcal{T}}(z):=z \exp(-Li_{2}(z)), 
\end{equation*}	
where $$-\int_{0}^{z}\frac{\log(1-t)}{t}dt= \sum_{n=1}^{\infty}\frac{z^n}{n^2}=:Li_{2}(z)$$ known as dilogarithm function and 
\begin{equation*}
f_{\mathcal{L}}(z):= z \exp \int_{0}^{z}\frac{1}{\cos{\beta t}}dt= z(\sec{\beta z}+\tan{\beta z})^{1/\beta}),\quad \beta\neq0.
\end{equation*}

\begin{figure}[h]
	\begin{tabular}{c}
		\includegraphics[scale=0.4]{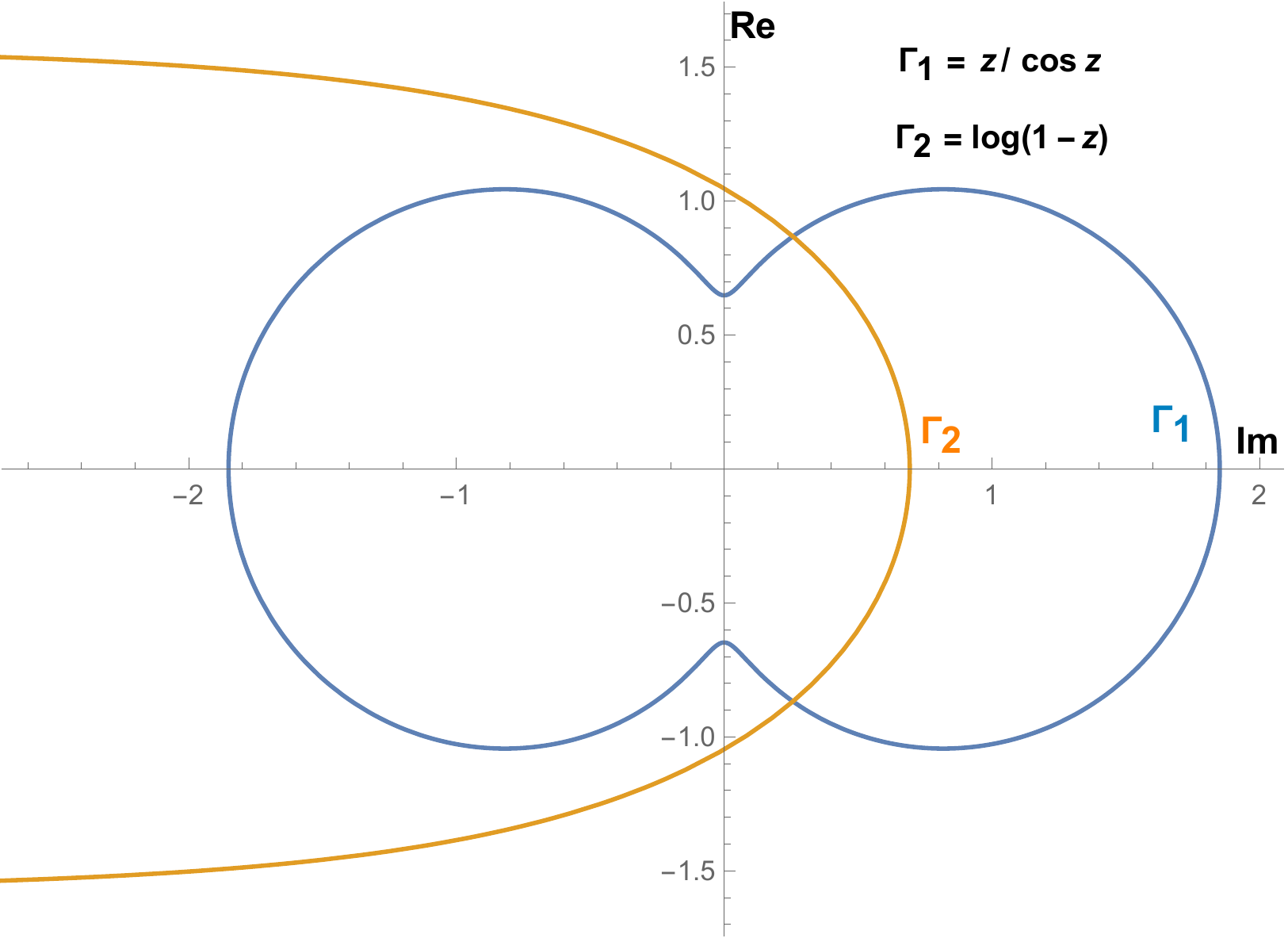}
	\end{tabular}
	\caption{Boundary curves of the functions $z/{\cos{z}}$ and $\log(1-z)$}\label{f1}
\end{figure}

\section*{Conclusion}
It is interesting to observe that even in the class $\mathcal{F}(\psi)$, functions may not be univalent. But with the conditions on the bounds for the real part of $\psi$, a similar result holds as obtained by Ma-Minda~\cite{minda94} which is quiet important to obtain the Koebe domain. From Remark~\ref{3.2} and Remark~\ref{Fk-convex}, we also note that the function $f_0(z)/z$, where $f_0$ as defined in \eqref{f0} behaves quite differently in the particular classes. 
	
\section*{Conflict of interest}
	The authors declare that they have no conflict of interest.

\end{document}